\documentclass[11pt,twoside]{article}
\usepackage{latexsym,amsmath,amssymb,amsthm,times,tikz,pgfplots}

\newif\ifdviwin

\renewenvironment{thebibliography}[1]{
  \begin{oldthebibliography}{#1}
    \setlength{\itemsep}{0em}
    \setlength{\parskip}{0em}
}
{
  \end{oldthebibliography}
}

\usepackage{titlesec}
	
\titleformat{\section}{\large\bfseries\center}{\thesection}{1em}{\vspace{.0cm}}
\titleformat{\subsection}[runin]{\bfseries}{\thesubsection}{1em}{}

\usepackage[english]{babel}
\usepackage{indentfirst}
\usepackage[mathscr]{eucal}
\usepackage{amssymb,amsmath,amsfonts}
\usepackage{fancybox,fancyhdr}
\usepackage{graphicx}
\usepackage[utf8]{inputenc}
\usepackage{float}
\usepackage{color}
\usepackage[pdftex]{hyperref}
\hypersetup{colorlinks=true,linkcolor=blue,citecolor=blue}

\newif\ifdviwin

\dviwintrue

\newtheorem{theorem}{Theorem}[section]
\newtheorem{corollary}[theorem]{Corollary}
\newtheorem{proposition}[theorem]{Proposition}

\theoremstyle{definition}
\newtheorem{definition}[theorem]{Definition}

\def\r{\mathbb R}
\def\R{\mathcal{R}}
\def\H{\mathbb H}

\def\s{\mathbb S}
\def\n{\mathbf n}

\def\h{\mathsf h}
\def\acot{\mathrm{arccot}}

\begin{document}
\thispagestyle{empty}

\begin{center}

\renewcommand{\thefootnote}{\,}
{\Large \bf Invariant $\lambda$-translators in $\mathbb{S}^2\times\mathbb{R}$
\footnote{\hspace{-.75cm}
\emph{Mathematics Subject Classification:} \\
\emph{Keywords}: }}\\
\vspace{0.5cm} { Antonio Bueno}\\
\end{center}
\vspace{.5cm}

\begin{abstract}
A $\lambda$-translator in $\mathbb{S}^2\times\mathbb{R}$ is an oriented surface whose mean curvature $H$ satisfies $H=\langle N,\partial_z\rangle+\lambda$, where $N$ is the unit normal, $\partial_z$ is the vertical Killing vector field tangent to the fibers of the submersion and $\lambda\in\r$. When $\lambda=0$ we fall into the class of translators. In this paper, we classify all $\lambda$-translators that are invariant by a one-parameter group of rotations and by vertical translations of $\s^2\times\r$.
\end{abstract}


\section{Introduction}

In recent years, the study of surfaces in the Euclidean 3-space $\r^3$ endowed with a smooth density has become an active and fruitful field of research. The setting is as follows: consider $e^\phi$ a smooth density, where $\phi\in C^\infty(\r^3)$, which serves as weight to measure the surface area and volume, $dA_\phi=e^\phi dA$ and $dV_\phi=e^\phi dV$. If we consider a compactly supported variation of $M$ with variation vector field $\xi$, then
$$
A'_\phi(0)=\int_MH_\phi\langle N,\xi\rangle dA_\phi,\quad V'_\phi(0)=-\int_M\langle N,\xi\rangle,
$$
where $H_\phi=H-\langle N,D\phi\rangle$ and $D$ is the gradient in $\r^3$. Thus, $M$ is a critical point of the weighted area for a compactly supported variation that preserves the weighted volume if and only if $H_\phi$ is a constant function, $H_\phi=\lambda$ \cite{bcmr}. If we drop the condition that the variations preserve the volume, then $M$ is a critical point of $A_\phi$ if and only if $H_\phi=0$. Some authors call the latter \emph{$\phi$-minimal surfaces}. Particular choices of $\phi$ that give rise to well-known $\phi$-minimal surfaces are the following:
\begin{enumerate}
\item Minimal surfaces for $\phi$ a constant function. As extensive as this theory is, we refer the reader to \cite{mp} to a complete survey on classical and recent advances on this theory. 
\item Self-shrinkers (resp. self-expanders) for $\phi(x)=-|x|^2/4$ (resp. for $\phi=|x|^2/4$) \cite{ilm,smo}.
\item Singular minimal surfaces for $\phi(x)=\log\langle x,\textbf{v}\rangle$, where $\textbf{v}\in\r^3$ \cite{die,lop3}.
\item Translating solitons of the mean curvature flow (MCF) for $\phi(x)=\langle x,\textbf{v}\rangle$. We will refer simply as \emph{translators}. Without aiming to collect all the bibliography, we refer the reader to \cite{css,himw,sx} and references therein for an outline of the development of this theory.
\end{enumerate}
Among the aforementioned class of surfaces arising for different densities, we make special emphasis to the one defined by the linear density $\phi(x)=\langle x,\textbf{v}\rangle$. Translators appear in the singularity theory of the MCF as the equation of the limit flow by a blow-up procedure near type II singularities; see \cite{hui,hs}, hence they are solutions of the MCF that evolve purely by translations along the direction $\textbf{v}$. We want to remark that equation $H_\phi=0$ appeared in the classical article of Serrin \cite{ser} and it was studied in the context of the maximum principle of elliptic equations. 

Besides the Euclidean 3-space, translators have been also generalized to homogeneous 3-dimensional manifolds of great interest: the hyperbolic 3-space \cite{bulo3,lrs}; the product  spaces $\mathbb{H}^2\times\r$ \cite{bue1,bue2,bulo1,lipi} and $\s^2\times\r$ \cite{lm1}; the Heisenberg and the solvable group \cite{pip1,pip2}; and the special lineal group $SL(2,\r)$ \cite{lm2}. At each of such spaces, a translator is defined by substituting in Eq. \eqref{eqlambdatranslator} $\textbf{v}$ by a Killing vector field $X$, hence its mean curvature is given by $H=\langle X,N\rangle$. We emphasize that also conformal Killing vector fields have been considered, \cite{bulo2,moshs}, with the difference that the shape of the corresponding translators is not preserved along the flow. Finally, translators have been also addressed in Lorentzian spaces \cite{laor}.

In contrast with the fruitful theory of translators developed in homogeneous 3-manifolds, a systematic study of the equation $H_\phi=\lambda$, with $\lambda\neq0$, has been only considered in $\r^3$ \cite{blo,buor1,lop1,lop2}, in $\H^2\times\r$ \cite{buor2} and in $\mathbb{L}^3$ \cite{buor3}. Our objective is to cover these gaps in the remaining 3-dimensional Thurston geometries and as starting point we take the product space $\s^2\times\r$. The product structure decomposes the space of Killing vector fields as the ones corresponding to the base $\s^2$ and the translations in the $\r$-factor. Hence the Killing vector fields of $\s^2\times\r$ are the three infinitesimal generators of rotations of $\s^2$, leaving the $\r$-factor invariant, and translations in the $\r$-direction. If we consider coordinates $(p,z)\in\s^2\times\r$, translations in the $\r$-factor are given by $T_t(p,z)=(p,z+t),\ t\in\r$. Furthermore, for fixed $(p,z)$, the curve $\alpha(t)=(p,z+t)$ is a geodesic which is the integral line of the Killing vector field $\partial_z=\alpha'(t)=(0,1)$. In this paper, we take $\partial_z$ as the Killing vector field that generalizes the translator equation and consider the following surfaces.

\begin{definition}
Given $\lambda\in\r$, an oriented surface $M$ in $\s^2\times\r$ is a $\lambda$-translator if its mean curvature $H$ satisfies
\begin{equation}\label{eqlambdatranslator}
H(p)=\langle N(p),\partial_z\rangle+\lambda,\qquad p\in M,
\end{equation}
where $N$ is the unit normal vector field of $M$.
\end{definition}
In this paper we always assume $\lambda\neq0$ in order to avoid the translator case studied in \cite{lm1}. The product structure makes readily available two canonical projections,
$$
\pi:\s^2\times\r\to\s^2,\qquad \h:\s^2\times\r\to\r.
$$
The former is a Riemannian submersion whose \emph{fibers}, defined as the preimages $\pi^{-1}(z)$, are precisely geodesics tangent to $\partial_z$. The latter is known as the \emph{height function} and satisfies $D\h=\partial_z$, where $D$ is now the gradient in $\s^2\times\r$. Consequently, $\lambda$-translators are also critical points for the weighted variational problem for the density $e^z$.

Some examples of $\lambda$-translators are the following:
\begin{enumerate}
\item Let $\gamma$ be a geodesic in $\s^2$. The \emph{vertical plane} over $\gamma$ is the surface $P_\gamma=\pi^{-1}(\gamma)=\gamma\times\r$. If $\n_\gamma$ a unit normal along $\gamma$ in $\s^2$, then $N_\gamma=(\n_\gamma,0)$ is a unit normal vector field on $P_\gamma$. Such a surface is totally geodesic and is a $\lambda$-translator for the orientations $\pm N_\gamma$ if and only if $\lambda=0$. In other words, $P_\gamma$ is a translator.
\item Let $\gamma$ be a circle in $\s^2$ with constant geodesic curvature $\kappa_\gamma$. The \emph{vertical cylinder} over $\gamma$ is the surface $C_\gamma=\pi^{-1}(\gamma)$, which has constant mean curvature $H=\kappa_\gamma/2$. As for vertical planes, vertical cylinders are everywhere tangent to $\partial_z$ hence their weighted mean curvature is $H_\phi=H$. Consequently, $C_\gamma$ is a $\lambda$-translator if and only if $\kappa_\gamma=2\lambda$.
\item Given $z_0\in\r$, the \emph{horizontal plane} at height $z_0$ is $\Pi_{z_0}=\h^{-1}(z_0)=\s^2\times\{z_0\}$. Such a surface is totally geodesic and is a $\lambda$-translator for the orientation $\epsilon\partial_z$ if and only if $\lambda=-\epsilon$, where $\epsilon\in\{\pm1\}$.
\end{enumerate}

We focus on classifying invariant $\lambda$-translators, as well as in obtaining some uniqueness and non-existence results. To be precise, fix $p_0\in\s^2$ and consider the geodesic $L_0=\{p_0\}\times\r$. In this paper, we classify $\lambda$-translators that are invariant under rotations about $L_0$ and under translations parallel to $L_0$. This follows the same scheme as most of the aforementioned works regarding translators and $\lambda$-translators in different homogeneous spaces, since invariant $\lambda$-translators are commonly used as barriers to prove non-existence and uniqueness results. Specifically, the two main classification results are the following.

\begin{theorem}\label{t1}
For each $\lambda\in\r$, there are exactly two rotational $\lambda$-translators intersecting orthogonally the rotation axis, denoted by $M_\pm$, whose unit normal at such intersection point is $\pm\partial_z$. Furthermore, $M_+$ is a properly embedded disk whose end converges to the vertical cylinder of radius $\acot(2\lambda)$, and 
\begin{enumerate}
\item If $\lambda>1$, then $M_-$ is a properly immersed disk that self-intersects as it loops a finite number of times until converging the vertical cylinder of radius $\pi-\acot(2\lambda)$.
\item If $\lambda=1$, then $M_-$ is a horizontal slice $\s^2\times\{z_0\}$, $z_0\in\r$.
\item If $\lambda<1$, then $M_-$ is a properly embedded disk whose end converges to the vertical cylinder of radius $\pi-\acot(2\lambda)$.
\end{enumerate}
\end{theorem}

\begin{theorem}\label{t2}
For any $\lambda>0$, a complete, rotational $\lambda$-translators not intersecting the rotation axis is a properly immersed annulus that loops a finite number of times, and whose ends converge to the CMC cylinders generated by $e_0$ and $e_1$. 

Moreover, if $\lambda\leq1$ then the number of loops is exactly one, and if $\lambda>1$ then the number of loops is at least the number of loops of $M_-$.
\end{theorem}

Next, we detail the organization of the paper and highlight some of the main results. In Section \ref{s1} we introduce the basic notation and announce in Thm. \ref{thcomparsiontangency} the comparison and tangency principles in geometric terms, which are a consequence of the fact that Eq. \eqref{eqlambdatranslator} is of divergence type.

\section{Preliminaries and first results}\label{s1}

We regard the 2-sphere in $\r^3$ as the subset $\s^2=\{(x_1,x_2,x_3)\in\r^3\colon x_1^2+x_2^2+x_3^2=1\}$, hence the space $\s^2\times\r$ is isometrically immersed in $\r^4$ as 
$$
\s^2\times\r=\{(x_1,x_2,x_3,z)\in\r^4\colon x_1^2+x_2^2+x_3^2=1\}.
$$
There are two projections: the height function $\h(p,z)=z$ defined in the introduction and the projection $\pi:\s^2\times\r\to\s^2,\ \pi(x_1,x_2,x_3,z)=(x_1,x_2,x_3)$. Then, $\pi$ is a Riemannian submersion whose fibers $\pi^{-1}(x_1,x_2,x_3)$ are geodesics, and translations along these geodesics are induced as the flow of the Killing vector field $\partial_z$. 

Let us express Eq. \eqref{eqlambdatranslator} in non-parametric form. Let be $\Omega\subset\s^2$ and $u:\Omega\to\r$, and define $M=\{(p,u(p))\colon p\in\Omega\}$ the vertical graph of $u$. Denote by $\nabla$ and $\mathrm{div}$ the gradient and divergence operators in $\s^2$. The unit normal $N$ and the mean curvature of $M$ are
$$
N=\frac{1}{\sqrt{1+|\nabla u|^2}}((-\nabla u,0)+\partial_z),\qquad 2H=\mathrm{div}\frac{\nabla u}{\sqrt{1+|\nabla u|^2}},
$$
thus, Eq. \eqref{eqlambdatranslator}, or equivalently $H_\phi=\lambda$, is written as
\begin{equation}\label{eqlambdatranslatorgrafo}
\mathrm{div}\frac{\nabla u}{\sqrt{1+|\nabla u|^2}}=2\left(\frac{1}{\sqrt{1+|\nabla u|^2}}+\lambda\right).
\end{equation}
This equation is elliptic and of divergence type, hence the comparison and maximum principles of quasilinear, elliptic equations apply, see e.g. \cite{gitu}, Chapter 10. We formulate them in geometric terms. Let $M_1$ and $M_2$ be two surfaces in $\s^2\times\r$, possibly with boundaries $\partial M_1$ and $\partial M_2$. Let $p\in M_1\cap M_2$ be a common tangent point and take orientations $N_1,N_2$ such that $N_1(p)=N_2(p)$. If $p\in \partial M_1\cap\partial M_2$ we also assume that $\partial M_1$ and $\partial M_2$ are tangent at $p$. Around $p$, consider $M_1$ and $M_2$ as graphs of functions $u_i:\Omega\to\r$, where $\Omega$ is a domain of the tangent plane and fix $N_i(p)$ as the positive direction. With this reference system, we say that $M_1$ lies over $M_2$ around $p$, denoted by $M_1\geq M_2$, if $u_1\geq u_2$ in $\Omega$.

\begin{theorem}\label{thcomparsiontangency}
 Let $M_1$ and $M_2$ be two surfaces in $\s^2\times\r$ with weighted mean curvatures   $H_\phi^{M_1}$ and $H^{M_2}_\phi$, respectively. 
\begin{enumerate}
\item (Comparison principle) If $M_1\geq M_2$ around $p$, then $H_\phi^{M_1}(p)\geq H_\phi^{M_2}(p)$. 
\item (Tangency principle) Assume that $H_\phi^{M_1}$ and $H_\phi^{M_2}$ are constant and $H_\phi^{M_1}=H_\phi^{M_2}$. If $M_1\geq M_2$ around $p$, then $M_1$ and $M_2$ coincide in an open set around $p$.
\end{enumerate}
\end{theorem}

The comparison and tangency principles allow us to prove the following result.

\begin{theorem}\label{thnonexistenceclosed1}
Let $|\lambda|\leq1$. Then, there do not exist closed $\lambda$-translators.
\end{theorem}

\begin{proof}
The proof is by contradiction, so assume that $M$ is a closed $\lambda$-translator. Consider the family of horizontal planes $\Pi_t$, and recall that they are $\epsilon$-translators for the orientation $-\epsilon\partial_z$, $\epsilon=\pm1$. Let $p_0\in M$ the point where $\h$ attains a global maximum and let $z_0=\h(p_0)$, hence $N(p_0)=\pm\partial_z$. In this situation, $M$ and the horizontal plane $\Pi_{z_0}$ are tangent at $p_0$. We distinguish cases depending on the orientation of $M$ at $p_0$.

If $N(p_0)=\partial_z$, then we orient $\Pi_{z_0}$ by $\partial_z$ and hence $\Pi_{z_0}\geq M$ around $p_0$. The comparison principle yields
$$
-1=H_\phi^{\Pi_{z_0}}(p_0)\geq H_\phi^M(p_0)=\lambda.
$$
If $|\lambda|<1$ we arrive to a contradiction, thus the only possibility is $\lambda=-1$. However, $\Pi_{z_0}$ is a $-1$-translator for the orientation $\partial_z$ and this would yield
a contradiction with the tangency principle.

If $N(p_0)=-\partial_z$, then we orient $\Pi_{z_0}$ by $-\partial_z$ and now $M\geq\Pi_{z_0}$ around $p_0$, hence
$$
\lambda=H_\phi^M(p_0)\geq H_\phi^{\Pi_{z_0}}=1,
$$
and arguing as before necessarily $\lambda=1$. This contradicts again the tangency principle and we are done.
\end{proof}

The case $|\lambda|>1$ is trickier. In $\r^3$ and $\H^2\times\r$, the proof uses the following fact. Let $\mathbb{M}^2(\kappa)$ denote the complete, simply connected surface of constant curvature $\kappa$. A well-known formula states that if $M$ is a surface in $\mathbb{M}^2(\kappa)\times\r$ and we consider the restriction of $\h$ to $M$, then $\nabla\h=(\partial_z)^T=\partial_z-\langle N,\partial_z\rangle N$. Consequently, $\Delta\h=2H\langle N,\partial_z\rangle$, and if $M$ is a $\lambda$-translator then
\begin{equation}\label{eqlaplacianoh}
\Delta\h=2\langle N,\partial_z\rangle^2+2\lambda\langle N,\partial_z\rangle.
\end{equation}
Thus, Stokes theorem yields that $M$ is included in a vertical cylinder $\gamma\times\r$, which contradicts its closeness. This proof strongly relies in the fact that any closed surface is a 2-cycle that encloses a bounded 3-chain, hence Stokes theorem can be applied to prove that $\int_M\langle N,\partial_z\rangle=0$.

Although Eq. \eqref{eqlaplacianoh} also holds in $\s^2\times\r$, a closed surface may not be the boundary of a 3-chain, as the topology of the base of $\s^2\times\r$ has further implications on the properties of closed surfaces. Indeed, its second homology group is $H_2(\s^2\times\r)=\mathbb{Z}$ and a generator of such group is any horizontal plane $\s^2\times\{z_0\}$, which is closed. Furthermore, we can orient $\s^2\times\{z_0\}$ by $\partial_z$, hence $\langle N,\partial_z\rangle=1$ and
$$
\int_{\s^2\times\{z_0\}}\langle N,\partial_z\rangle=4\pi.
$$
In the following result, which is valid for any $\lambda\in\r$, we make an extra assumption on the homology class of the $\lambda$-translator.

\begin{theorem}\label{thnonexistenceclosed2}
There do not exist closed, null-homologous $\lambda$-translators.
\end{theorem}

\begin{proof}
If $M$ is null-homologous in $H_2(\s^2\times\r)$ then $M$ is the boundary of a 3-chain $W$ and Stokes theorem applied to the vector field $\partial_z$ in $W$ yields that $\int_M\langle N,\partial_z\rangle=0$. Now, let us consider the restriction of the height function $\h$ to $M$ and integrate
$$
0=\int_M\Delta\h=2\int_M\langle N,\partial_z\rangle^2+2\lambda\int_M\langle N,\partial_z\rangle=2\int_M\langle N,\partial_z\rangle^2.
$$
This implies $\langle N,\partial_z\rangle=0$ in $M$ and thus $M=\gamma\times\r$ with $\gamma\subset\s^2$ a curve, a contradiction since $M$ is closed.
\end{proof}

The comparison and tangency principles can be used to prove a half-space type result for compact $\lambda$-translators with boundary in a horizontal plane. 
\begin{proposition}\label{propaunlado}
Let $|\lambda|\leq 1$ and $M$ be a compact $\lambda$-translator with boundary in a horizontal plane $\s^2\times\{z_0\}$. Then, $\max_M\h=z_0$.
\end{proposition}

\begin{proof}
Assume that there exists $p_0\in\mathrm{int}(M)$ such that $\h(p_0)>z_0$; without losing generality, assume that $\h_{|_M}$ attains its maximum at $p_0$. Then, $N(p_0)=\pm\partial_z$. If $N(p_0)=\partial_z$ then we orient $\Pi_{\h(p_0)}$ by $\partial_z$ and hence $\Pi_{\h(p_0)}\geq M$ around $p_0$. Since $H_\phi(p_0)=\lambda$ and $H_\phi^{\Pi_{\h(p_0)}}=-1$ the comparison principle yields $-1\geq\lambda$ and the assumption on $\lambda$ yields $\lambda=-1$. This is however a contradiction with the tangency principle. If $N(p_0)=-\partial_z$ then $H_\phi=\lambda$ and $H_\phi^{\Pi_{\h(p_0)}}=1$, and $M\geq\Pi_{\h(p_0)}$ around $p_0$. Again, the comparison principle and the assumption on $\lambda$ yields $\lambda=1$, which contradicts again the tangency principle.
\end{proof}

As a consequence of this result and since we can apply reflections about vertical planes since Eq. \eqref{eqlambdatranslator} is invariant under such ambient isometries, we obtain the following symmetry result. The proof strongly relies on Alexandrov reflection technique, see e.g. \cite{lop} for a detailed outline for the CMC case.
\begin{corollary}
Let $|\lambda|\leq1$ and $M$ be a compact $\lambda$-translator with boundary in a horizontal plane $\s^2\times\{z_0\}$. Assume that there exists a geodesic $\gamma\subset\s^2$ such that $\gamma\times\{z_0\}$ separates $\partial M$ into two connected components that are graphs over $\gamma\times\{z_0\}$. Then, $M$ is symmetric about the vertical plane $\gamma\times\r$. In particular, if $\partial M$ is a circle, then $M$ is rotational.
\end{corollary}

\begin{proof}
By Prop. \ref{propaunlado} we know that $M$ lies contained in the half-space $\{z\leq z_0\}$ determined by $\Pi_{z_0}$. This allows us to apply the classical Alexandrov reflection technique, as no accident occurs between an interior and boundary point. Furthemore, since $\partial M$ is a bi-graph over $\gamma\times\{z_0\}$ then no accident occurs between two non-tangent boundary points. Thus, Alexandrov reflection technique yields that $\gamma\times\r$ is a plane of symmetry of $M$. In the case that $\partial M$ is a circle, then any vertical plane containing its center is a plane of symmetry of $M$ and thus $M$ is rotational.
\end{proof}
Obviously, if $|\lambda|>1$ and we make the extra assumption that $M$ lies contained in one of the half-spaces determined by the plane where its boundary lies, the result above still holds.

We finish this section by classifying all $\lambda$-translators invariant by vertical translations.

\begin{proposition}
Suppose that $M$ is a surface invariant by the group of vertical translations generated by $\partial_z$. Then, $M$ is a $\lambda$-translator if and only if its generating curve is a closed circle in $\s^2$ of constant geodesic curvature $\kappa_g=2\lambda$.
\end{proposition}

\begin{proof}
Let $M$ a surface invariant by the vertical translations generated by $\partial_z$. In particular, the fibers of the submersion $\pi:\s^2\times\r\to\s^2$ are tangent to $M$ and thus $M=\gamma\times\r$, where $\gamma\subset\s^2$ is a curve. In such a case, the mean curvature $H$ of $M$ is $H=\kappa_\gamma/2$. Since $M$ is tangent to $\partial_z$, one has $\langle N,\partial_z\rangle=0$ and hence $H=\lambda$. This implies $\kappa_\gamma=2\lambda$ and consquently $\gamma$ is a curve of constant geodesic curvature in $\s^2$, which proves the result.
\end{proof}

\section{Rotational $\lambda$-translators}\label{s2}

In this section we derive some formulae regarding rotational $\lambda$-translators and prove the existence of radial solutions of Eq. \eqref{eqlambdatranslator}. Without losing generality, we choose as rotation axis the fiber passing through the north pole, $\{(0,0,1)\}\times\r$, hence a rotational surface in $\s^2\times\r$ is locally parametrized by
$$
\psi(s,\phi)=(\sin x(s)\cos\phi,\sin x(s)\sin\phi,\cos x(s),z(s)),\qquad s\in I,\ \phi\in\r.
$$
The function $x(s)$ has range $x(s)\in(0,\pi)$ and is the distance to the rotation axis. The function $z(s)$ agrees with the restriction of the height function to $M$. The curve $\beta(s)=(x(s),z(s))$ is assumed to be arc-length parametrized, hence $x'(s)=\cos\theta(s),\ z'(s)=\sin\theta(s)$ for a smooth function $\theta(s)$. Note that $\theta'(s)$ is just the curvature of $\beta$ as a planar curve of $\r^2$. The unit normal is 
$$
N=(-\cos x(s)\sin\theta(s)\cos \phi,-\cos x(s)\sin\theta(s)\sin \phi,\sin x(s)\sin\theta(s),\cos\theta(s)),
$$
and hence $\langle N,\partial_z\rangle=\cos\theta(s)$. From now on, we omit the dependence on the variable $s$ unless explicitly necessary. The mean curvature is $2H=\theta'+\sin\theta \cot x$ and the condition of being a $\lambda$-translator is $H=\cos\theta+\lambda$. We conclude that the following system must be fulfilled
\begin{equation}\label{eqsystem}
\left\lbrace
\begin{array}{l}
x'=\cos\theta,\\
z'=\sin\theta,\\
\theta'=2(\cos\theta+\lambda)-\sin\theta\cot x.
\end{array}
\right.
\end{equation}
A particular case of \eqref{eqsystem} appears when $\theta$ is constant, hence $\beta(s)=(x(s),z(s))$ is a straight line and the angle $\langle N,\partial_z\rangle$ is constant. From the first equation in \eqref{eqsystem}, we deduce $x(s)=s\cos\theta+a,\ a\in\r$, and substituting in the third equation of \eqref{eqsystem} yields
$$
2(\cos\theta+\lambda)=\sin\theta\cot(s\cos\theta+a).
$$
Since the left-hand side is constant so it must be the right-hand side, hence $\cos\theta=0$ and thus $\theta=\pi/2+k\pi,\ k\in\mathbb{Z}$, which yields $x(s)=\acot(2\lambda)$ and $M$ is a CMC cylinder of radius $\acot(2\lambda)$. We have proved the following result.

\begin{proposition}
The only rotational $\lambda$-translators making a constant angle with the vector field $\partial_z$ are circular cylinders of radius $\acot(2\lambda)$.
\end{proposition}

Now, we prove that if a rotational $\lambda$-translator meets the rotation axis, it does so orthogonally.

\begin{proposition}\label{proporthogonalintersection}
If the profile curve of a rotational $\lambda$-translator intersects the rotation axis, then it does so at an orthogonal angle.
\end{proposition}

\begin{proof}
From the equations of system \eqref{eqsystem}, we have
\begin{align*}
(xz')'&=x'z'+xz''=x'z'+xx'\theta'\\
&=x'z'+xx'(2(x'+\lambda)-z'\cot x)\\
&=x'z'(1-x\cot x)+2xx'^2+2\lambda xx'.
\end{align*}
For fixed $s_0\in\r$, we integrate from $s_0$ to $s$, arriving to
$$
x(s)\sin\theta(s)-x(s_0)\sin\theta(s_0)=\int_{s_0}^sx'(t)z'(t)(1-x(t)\cot x(t))+2x(t)x'(t)^2\,dt+\lambda(x(s)^2-x(s_0)^2).
$$
Now, assume that the intersection with the rotation axis is at $s_0$, hence $x(s_0)=0$ and consequently,
\begin{equation}\label{eqfirstintegral}
x(s)\sin\theta(s)-\lambda x(s)^2=\int_{s_0}^sx'(t)z'(t)(1-x(t)\cot x(t))+2x(t)x'(t)^2\,dt.
\end{equation}
Let $f(s)=1-x(s)\cot x(s)$ and recall that $f$ extends continuously to $s=s_0$ by $f(s_0)=0$. For $s\neq s_0$ close enough to $s_0$ we have $x(s)>0$. Dividing the above equation by $x(s)$ we have
$$
\sin\theta(s)-\lambda x(s)=\frac{1}{x(s)}\int_{s_0}^sx'(t)z'(t)f(t)+2x(t)x'(t)^2\,dt.
$$
Now we let $s\to s_0$ and apply the L'Hôpital rule in the right-hand side, obtaining
$$
\sin\theta(s_0)=\sin\theta(s_0)f(s_0)=0,
$$
concluding $\theta(s_0)=k\pi$ and therefore the intersection is orthogonal.
\end{proof}

If $M$ is a $\lambda$-translator for the orientation $N$, then $M$ is a $-\lambda$-translator for the orientation $-N$. Consequently and up to a change of the orientation, \emph{hereinafter $\lambda$ will be assumed to be positive}.

\subsection{Existence of radial solutions}

In Prop. \ref{proporthogonalintersection} we proved that any rotational $\lambda$-translator approaching the rotation axis must do it orthogonally; in particular, if an intersection occurs. To prove the existence of a rotational $\lambda$-translator intersecting the rotation axis we cannot invoke standard existence theory since system \eqref{eqsystem} is singular at $x=0$. Nonetheless, in \cite{bue3} we study prescribed mean curvature graphs in homogeneous spaces and in particular in $\s^2\times\r$. In particular, we prove the existence of radial solutions to the prescribed mean curvature equation over a small-enough disk and in particular we prove the existence and uniqueness of two $\lambda$-translators intersecting orthogonally the rotation axis with unit normal $\pm\partial_z$, respectively, at such intersection.

\begin{definition}\label{defM+-}
We define $M_+$ (resp. $M_-$) the rotational $\lambda$-translator interesecting orthogonally the rotation axis with upwards (resp. downwards) orientation.
\end{definition}

Let us focus on the surfaces $M_+$ and $M_-$ and the third equation of \eqref{eqsystem},
$$
\theta'=2\lambda+2\cos\theta-\sin\theta\cot x.
$$
\begin{enumerate}
\item If $\lambda=1$, a trivial solution is $x(s)=s,\ z(s)=z_0\in\r$ and $\theta(s)=\pi$, that is $\alpha(s)$ is a great circle in $\s^2$ and the corresponding $\lambda$-translator $M_-$ is the horizontal slice $\s^2\times\{z_0\}$.
\item For the surface $M_+$ with the upwards orientation, we know that $u''(0)=1+\lambda$ and hence the surface $M_+$ is always strictly convex around the rotation axis. For the surface $M_-$ with the downwards orientation, we know that $u''(0)=1-\lambda$. If $\lambda=1$ we know that $u$ is constant and $M_-=\s^2\times \{z_0\}$. If $\lambda>1$ then $u''(0)<0$ and $M_-$ bends towards its unit normal. If $\lambda<1$ then $u''(0)>0$ and $M_-$ bends opposite to its unit normal.
\item From \eqref{eqfirstintegral} we deduce $x(s)\sin\theta(s)-\lambda x(s)^2>0$ for every $s$ and hence $x(s)\in(0,1/\lambda)$.
\item From Eq. \eqref{eqfirstintegral} again and the fact that the function $x\mapsto 1-x\cot x$ is positive for $x\in(0,\pi)$, we deduce that $\theta$ cannot attain the value $\pi$. As a consequence, \emph{there do not exist rotational $\lambda$-translators with the topology of a sphere}. Similarly, $\theta$ does not attain again the value $0$ because if $s_0>0$ is the first time where $\theta(s_0)=0$, then $\theta'(s_1)\leq0$, but \eqref{eqsystem} gives $\theta'(s_1)=2(1+\lambda)$, a contradiction. In particular, $\theta$ is a bounded function $\theta(s)\in(0,\pi)$ and the solutions of \eqref{eqsystem} are defined for $s\in(0,\infty)$.
\end{enumerate}

\section{Rotational $\lambda$-translators}\label{s3}

The objective of this section is to classify the rotational $\lambda$-translators. For that, note that system \eqref{eqsystem} can be studied by the first and third equations, since the second one depends on the first and third. This encodes the fact that \eqref{eqlambdatranslator} is preserved when we move a $\lambda$-translator by vertical translations. Hence we focus on the 2-dimensional system
\begin{equation}\label{eqphaseplane}
\left\lbrace
\begin{array}{l}
x'=\cos\theta,\\
\theta'=2(\cos\theta+\lambda)-\sin\theta\cot x.
\end{array}
\right.
\end{equation}

The phase plane is the set
$$
\mathcal{R}=\{(x,\theta)\colon x\in(0,\pi),\ \theta\in\r\},
$$
and the orbits $\gamma(t)=(x(t),\theta(t))$ are the solutions of \eqref{eqphaseplane} when regarded in $\R$. The $2\pi$-periodicity of the functions involving $\theta$ implies that the structure of $\mathcal{R}$ is $2\pi$-periodic in the $\theta$-direction. Thus, when we describe the elements of $\mathcal{R}$, we must take into account that they are defined modulo a discrete translation of length $2k\pi$ in the $\theta$-direction. 

It is immediate to verify that if $\gamma(s)=(x(s),\theta(s))$ is a solution of \eqref{eqphaseplane}, then $\delta(s)=(\pi-x(-s),-\theta(-s))$ is also a solution. Geometrically, $\R$ is anti-symmetric with respect to the lines $x=\pi/2$ and $\theta=0$. By the $2\pi$-periodicity in the $\theta$-direction, this also hols for $\theta=2k\pi,\ k\in\mathbb{Z}$. In particular, all the geometric elements that describe $\R$ share this anti-symmetry. There are two equilibrium points,
$$
e_0=(\mathrm{arccot}(2\lambda),\pi/2),\qquad e_1^k=(\pi-\mathrm{arccot}(2\lambda),-\pi/2),
$$
which correspond to vertical cylinders of constant mean curvature $\lambda$. The motion of each orbit in $\R$ is determined by the signs of $x'$ and $\theta'$. We have $x'>0$ if $\theta\in(-\pi/2,\pi/2)$, and $x'<0$ if $\theta \in(-\pi,-\pi/2)\cup(\pi/2,\pi)$, modulo $2\pi$. For the sign of $\theta'$, we must first study whether $\theta'=0$. For that, we define
$$
\Gamma(\theta)=\acot\frac{2(\cos\theta+\lambda)}{\sin\theta}.
$$
Then, the $\theta$-coordinate of an orbit $\gamma$ has a critical point if and only if the coordinates of $\gamma$ are related by $x=\Gamma(\theta)$. At this point, we must say a word of caution. Since $x\in(0,\pi)$ and in particular $x>0$, we must take care with the negative values of $\Gamma(\theta)\in(-\pi/2,0)$. By the anti-symmetry, these values correspond to $x=\pi-\Gamma(\theta)$. Thus, we define the curve
$$
\Gamma=\R\cap\left\lbrace
\begin{array}{lcc}
(\Gamma(\theta),\theta) & \mathrm{if} & \Gamma(\theta)>0,\\
(\pi-\Gamma(\theta),\theta) & \mathrm{if} & \Gamma(\theta)<0
\end{array}
\right\rbrace.
$$
Therefore, $\Gamma$ and the lines $\theta=\pm\pi/2$ separate $\R$ into connected components where the coordinate functions of any orbit are strictly monotonous. See Fig. \ref{figphaseplane}.
\begin{figure}[h]
\begin{center}
\includegraphics[width=.4\textwidth]{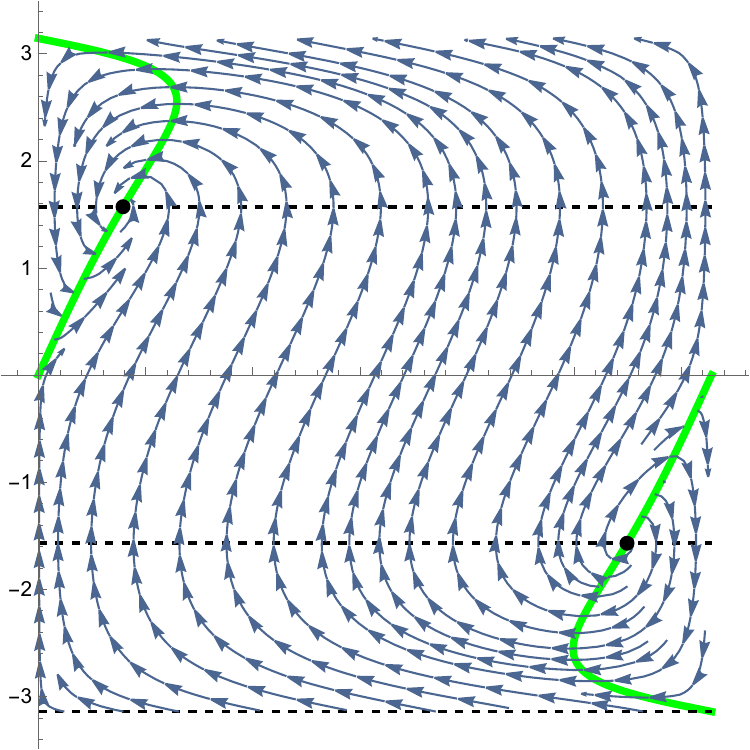}
\end{center}
\caption{The phase plane with the equilibrium points $e_0,e_1$ and the connected components of the curve $\Gamma$ that determines the monotonicity regions. Here, $\lambda>1$.}
\label{figphaseplane}
\end{figure}

As usual, we analyze the structure of the equilibrium points by linearizing the system. Let
$$
F(x,\theta)=\left(
\begin{matrix}
\cos\theta\\
2(\cos\theta+\lambda)-\sin\theta\cot x
\end{matrix}
\right).
$$
The jacobian matrix of $F$ evaluated at $e_0,e_1$ is
$$
JF(e_0)=\left(
\begin{matrix}
0&-1\\
1+4\lambda^2&-2
\end{matrix}
\right),\qquad JF(e_1)=\left(
\begin{matrix}
0&1\\
-1-4\lambda^2&2
\end{matrix}
\right),
$$
whose eigenvalues are $-1\pm 2\lambda i$ and $1\pm 2\lambda i$, respectively. Hence, the linearized structure around each equilibrium is as follows: the eigenvalues of $JF(e_0)$ are complex and of negative real part, hence $e_0$ is a stable spiral point and every orbit close enough to $e_0$ converges to it by spiraling around infinitely-many times. The eigenvalues of $F(e_1)$ are also complex but with positive real part, and thus $e_1$ is an unstable spiral point, hence every orbit close enough to $e_1$ \emph{escapes} from it.

The next result forbids the existence of closed orbits in $\R$. It follows from Bendixon-Dulac theorem, a classical result which appears in most textbooks on differential equations.
\begin{proposition}\label{propnonexistenceclosedorbits}
There do not exist closed orbits in $\mathcal{R}$.
\end{proposition}

\begin{proof}
Let us write system \eqref{eqphaseplane} as
$$
\left(
\begin{matrix}
x\\ \theta
\end{matrix}
\right)'=
\left(
\begin{matrix}
\cos\theta\\
2(\cos\theta+\lambda)-\sin\theta\cot x
\end{matrix}
\right)=
\left(
\begin{matrix}
P(x,\theta)\\Q(x,\theta)
\end{matrix}
\right),
$$
and define the function $\textbf{r}:\R\to\r$ and the vector field $V:\R\to\R$
$$
\textbf{r}(x,\theta)=\sin x,\qquad V(x,\theta)=\textbf{r}(x,\theta)(P(x,\theta),Q(x,\theta)).
$$
It is clear that $\mathrm{div} V=2\sin x$, which has constant sign in $\R$. Arguing by contradiction, assume that $\overline{\gamma}$ is a closed orbit in $\R$ and name $\Omega$ to its inner region. The divergence theorem yields
$$
0\neq\int_\Omega\mathrm{div}V=\int_{\overline{\gamma}}\langle V,\textbf{n}_{\overline{\gamma}}\rangle=0,
$$
where $\textbf{n}_{\overline{\gamma}}$ is the unit normal to $\overline{\gamma}$. Recall that the last integral vanishes since $V$ is everywhere tangent to $\overline{\gamma}$. This contradiction proves the result.
\end{proof}

As a consequence of this result and Poincaré-Bendixon theorem, we conclude the following.

\begin{corollary}
Let $\gamma$ be an orbit in $\R$, and assume that $\gamma(s)$ stays at a bounded distance of $e_k$ as $|s|\to\infty$. Then, $\gamma(s)\to e_0$ as $s\to\infty$, or $\gamma(s)\to e_1$ as $s\to-\infty$.
\end{corollary}
In particular, if an orbit starts to spiral around an equilibrium $e_k$, it must converge to it. We now announce the consequence in the phase plane of the existence of radial $\lambda$-translators intersecting orthogonally the rotation axis.

\begin{proposition}\label{propexistenceradial}
There exist two orbits, $\gamma_+$ and $\gamma_-$, corresponding to the rotational $\lambda$-translators $M_+$ and $M_-$, respectively, given by Def. \ref{defM+-}.
\end{proposition}

We finish this section with a technical result that will be useful in the sequel.
\begin{proposition}\label{propendpointsorbita}
Let $\gamma$ be an orbit in $\theta\in(-\pi,0)$ whose endpoints are $(x_1,0)$ and $(x_2,-\pi)$. Then, $x_1<x_2$.
\end{proposition}

\begin{proof}
We now that $\gamma$ can be expressed as a horizontal graph $x=x(\theta)$ with $x:[-\pi,0]\to\r$. Furthermore, the derivative $x'(\theta)$ is
\begin{equation}\label{eqderivadax}
x'(\theta)=\dfrac{\cos\theta}{2(\cos\theta+\lambda)-\sin\theta\cot x}.
\end{equation}
We decompose the function $x(\theta)$ as
$$
x_1:[-\pi/2,0]\to\r,\qquad x_2:[-\pi,-\pi/2]\to\r,
$$
which satisfy $x_1(-\pi/2)=x_2(-\pi/2)=x_0$, $x_1(0)=x_1$ and $x_2(-\pi)=x_2$. Furthermore, define
$$
\hat{x_2}:[-\pi/2,0]\to\r,\qquad \hat{x_2}(\theta)=x_2(-\theta-\pi).
$$
Then, the graph of $\hat{x_2}$ agrees with the one of $x_2$ up to a reflection about $\theta=-\pi/2$. Given $\epsilon>0$ we define $\theta_\epsilon=-\pi/2+\epsilon$. We have
\begin{align*}
x_1'(\theta_\epsilon)&=\dfrac{\cos(-\pi/2+\epsilon)}{2(\cos(-\pi/2+\epsilon)+\lambda)-\sin(-\pi/2+\epsilon)\cot x_1(\theta_\epsilon)}\\
&=\dfrac{\sin\epsilon}{2(\sin\epsilon+\lambda)+\cos\epsilon\cot x_1(\theta_\epsilon)},
\end{align*}
\begin{align*}
\hat{x_2}'(\theta_\epsilon)=-x_2'(-\pi/2-\epsilon)&=\dfrac{-\cos(-\pi/2-\epsilon)}{2(\cos(-\pi/2-\epsilon)+\lambda)-\sin(-\pi/2-\epsilon)\cot\hat{x_2}(-\pi/2-\epsilon)}\\
&=\dfrac{\sin\epsilon}{2(-\sin\epsilon+\lambda)+\cos\epsilon\cot\hat{x_2}(\theta_\epsilon)}.
\end{align*}
For $\epsilon$ small enough one has that $x_1(\theta_\epsilon)$ and $\hat{x_2}(\theta_\epsilon)$ are almost equal and thus $x_1'(\theta_\epsilon)<\hat{x_2}'(\theta_\epsilon)$, meaning that $\hat{x_2}$ lies locally above $x_1$ near $\theta_\epsilon$. 

Now we stand in position to prove the result.
Arguing by contradiction, if $x_1>x_2$ it means that the graph of $\hat{x_2}$ crosses the graph of $x_1$ transversely. Thus, there exists $\theta_*\in(-\pi/2,0)$ such that $x_1(\theta_*)=\hat{x_2}(\theta_*)=x_*$ and $x_1$ and $\hat{x_2}$ are transverse at $(x_*,\theta_*)$. We can assume without losing generality that $\theta_*$ is the largest value in $(-\pi/2,0)$ for which $x_1$ and $\hat{x_2}$ are transverse. Hence, there exists $\epsilon>0$ such that 
$$
x_1(\theta)<\hat{x_2}(\theta),\ \forall\theta\in(\theta_*-\epsilon,\theta_*),\quad x_1(\theta)>\hat{x_2}(\theta),\ \forall\theta\in(\theta_*,\theta_*+\epsilon).
$$
In particular, $x_1'(\theta_*)\geq\hat{x_2}'(\theta_*)$. However, this is a contradiction since
$$
x_1'(\theta_*)=\dfrac{\cos\theta_*}{2(\cos\theta+\lambda)-\sin\theta\cot x_*}<\dfrac{\cos\theta_*}{2(-\cos\theta+\lambda)-\sin\theta\cot x_*}=\hat{x_2}'(\theta_*).
$$
In the case that $x_1=x_2$ and $x_1$ and $\hat{x_2}$ never intersect until reaching the line $\theta=0$, another comparison between the derivatives yields $x_1'(0)\geq\hat{x_2}'(0)$. However, this is also a contradiction since
$$
x_1'(0)=\frac{1}{2(1+\lambda)},\qquad \hat{x_2}'(0)=-x_2'(-\pi)=\frac{1}{2(-1+\lambda)}.
$$
\end{proof}
By the anti-symmetry of the phase plane with respect to the axes $\{x=\pi/2\}\times\{\theta=0\}$, the opposite holds in the region $\theta\in(0,\pi)$, that is an orbit with endpoints $(x_1,\pi)$ and $(x_2,0)$ satisfy $x_1<x_2$.

\subsection{Proof of Thm. \ref{t1}}

\begin{proof}

We distinguish cases on the values of $\lambda$.

\textbf{Case $\lambda>1$}. In such a case, $\Gamma$ consists on two compact arcs, anti-symmetric with respect to the vertical line $x=\pi/2$. The first joins the points $(0,0)$ and $(0,\pi)$ and has a global maximum at $\theta=\arccos\frac{-1}{\lambda}=\pi-\arccos\frac{1}{\lambda}$, while the other joins $(\pi,0)$ and $(\pi,-\pi)$ and has the local maximum at $\theta=-\pi+\arccos\frac{1}{\lambda}$. See Fig. \ref{figphaseplanelambda>1}.

\begin{figure}[h]
\begin{center}
\includegraphics[width=.4\textwidth]{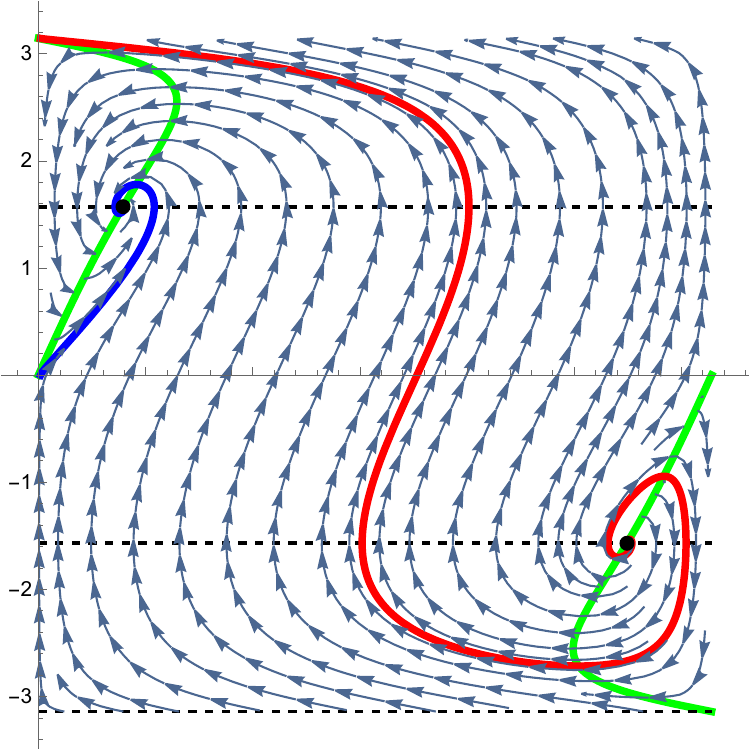}
\end{center}
\caption{The phase plane for $\lambda>1$. The curve $\Gamma$ has been plotted in green; the orbit $\gamma_+$ has been plotted in blue; and the orbit $\gamma_-$ appears in red. Here, $\lambda=1'2$}
\label{figphaseplanelambda>1}
\end{figure}

We begin by describing the orbits corresponding to the surfaces $M_\pm$. From Prop. \ref{propexistenceradial} we know the existence of an orbit $\gamma_+(s)=(x(s),\theta(s))$ such that $\gamma_+(0)=(0,0)$ and $\gamma_+(s)$ lying in the region $x>\Gamma(\theta)$ and $\theta>0$ for $s>0$ small enough. Similarly, there exists $\gamma_-$ such that $\gamma_-(0)=(0,\pi)$ and for $s<0$ small enough $\gamma_-(s)$ lies in the region $x>\Gamma(\theta)$ and $\theta<\pi$. See Fig. \ref{figphaseplanelambda>1}.

Let us focus on $\gamma_+$. Since the $x$-coordinate is bounded by $1/\lambda$, by monotonicity and properness of $\gamma_+$ there exists $s_+>0$ such that $\gamma(s_+)=(x(s_+),\pi/2)$. Similarly for $\gamma_-$, there exists $s_-<0$ such that $\gamma_-(s_-)=(x(s_-),\pi/2)$. We claim that $x_+=x(s_+)<x_-=x(s_-)$. Arguing by contradiction, assume that $x_+\geq x_-$. If $x_+=x_-$ then $\gamma_+$ and $\gamma_-$ can be smoothly glued together to define a larger orbit $\gamma_0$ that joins the points $(0,0)$ and $(0,\pi)$. It is clear that the rotational $\lambda$-translator $M_0$ defined by $\gamma_0$ is embedded and has the topology of a sphere. Furtheremore, since the $x$-coordinate is bounded in particular $M_0$ is the boundary of a bounded domain $W_0$ of $\s^2\times\r$. However, this is a contradiction with Thm. \ref{thnonexistenceclosed2}. It remains to prove that $x_+>x_1$ cannot happen. Arguing by contradiction, and since $\gamma_+$ and $\gamma_-$ cannot intersect each other, $\gamma_-$ has to intersect the curve $\Gamma$. Also, $\gamma_-$ cannot converge to a point $(0,\theta_0)$ with $\theta_0\in(0,\pi/2)$ by Prop. \ref{proporthogonalintersection} and thus $\gamma_-$ has to intersect again the line $\theta=\pi/2$. Finally, $\gamma_-$ has to intersect again the curve $\Gamma$ and since it cannot self-intersect, it must reach again the line $\theta=\pi/2$ at some $\hat{x_-}<x_-$. This process should be repeated and since there are not closed orbits in $\R$ thanks to \ref{propnonexistenceclosedorbits}, $\gamma_-$ should eventually converge to the equilibrium $e_0$ as $s\to-\infty$. This is a contradiction with the \emph{sink} structure of $e_0$, as it has a stable spiral structure.

Thus, $x_+<x_-$ and because $\gamma_+$ cannot intersect $\gamma_-$, we conclude that $\gamma_+$ must intersect the curve $\Gamma$. An argument similar as above yields that $\gamma_+$ must spiral around $e_0$. Since there do not exist closed orbits by Prop. \ref{propnonexistenceclosedorbits}, then $\gamma_+$ eventually converges to $e_0$ as $s\to\infty$; see Fig. \ref{figphaseplanelambda>1}, the orbit in blue. The corresponding surface $M_+$ is properly embedded, since the $z$-coordinate is monotonous, has the topology of a disk and its end converges to the CMC cylinder generated by $e_0$. See Fig. \ref{figperfiles}, left, the profile curve in blue.

\begin{figure}[h]
\begin{center}
\includegraphics[width=.8\textwidth]{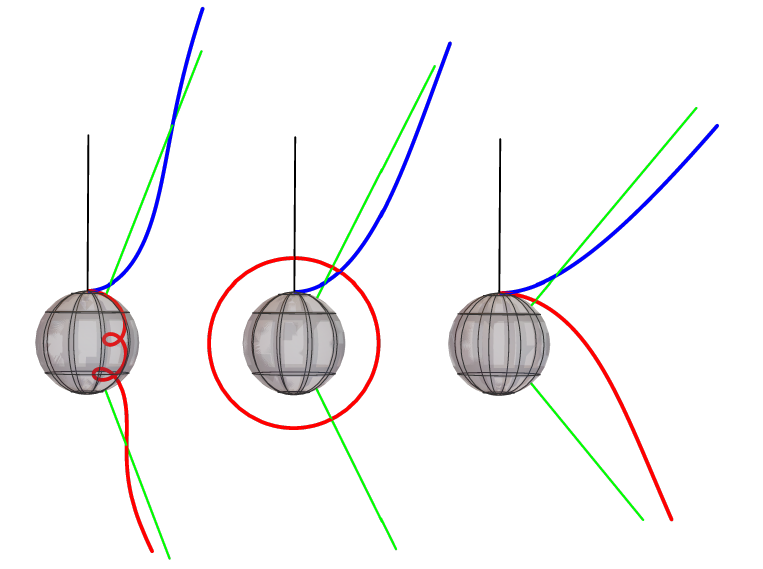}
\caption{The profile curve of a rotational $\lambda$-translator in a projected model of $\s^2\times\r$. Left, $\lambda>1$; center, $\lambda=1$; and right, $\lambda<1$.}
\label{figperfiles}
\end{center}
\end{figure}

Now we focus on $\gamma_-$. By monotonicity, $\gamma_-$ must intersect the line $\theta=0$. Now, the behavior of $\gamma_-$ changes depending on the value of $\lambda$. Indeed, following the motion in $\R$, it may intersect the curve $\Gamma$ in the region $\theta\in(-\pi,0)$, or may reach the line $\theta=-\pi$ at some $(x_1,-\pi)$. In the first case, $\gamma_-(s)\to e_1$ as $s\to-\infty$ and thus $\gamma_-$ \emph{emerges} from $e_1$; see Fig. \ref{figphaseplanelambda>1}, the orbit in red. Thus, we focus on the second case.

Let us consider the orbit $\sigma_1$ that passes through the point $(x_1,\pi)$. By uniqueness, $\sigma_1$ cannot intersect $\gamma_-$ in $\R$ and thus either $\sigma_1$ converges to $e_1$ as $s\to-\infty$ (see Fig. \ref{figcase>1}, left, the orbit in orange), or $\sigma_1$ reaches the line $\theta=-\pi$ at some $(x_2,-\pi)$ with $x_2>x_1$ (see Fig. \ref{figcase>1}, right, the orbit in orange). Now, we take advantage of the $2\pi$-periodicity of $\R$ in the $\theta$-direction and consider the translation $\sigma_1-(0,2\pi)$. By uniqueness and $2\pi$-periodicity, $\sigma_1-(0,2\pi)$ and $\gamma_-$ can be smoothly glued together at $(x_1,-\pi)$ to form a larger orbit, that obviously agrees with $\gamma_-$ and hence will be denoted the same. 

\begin{figure}[h]
\begin{center}
\includegraphics[width=.4\textwidth]{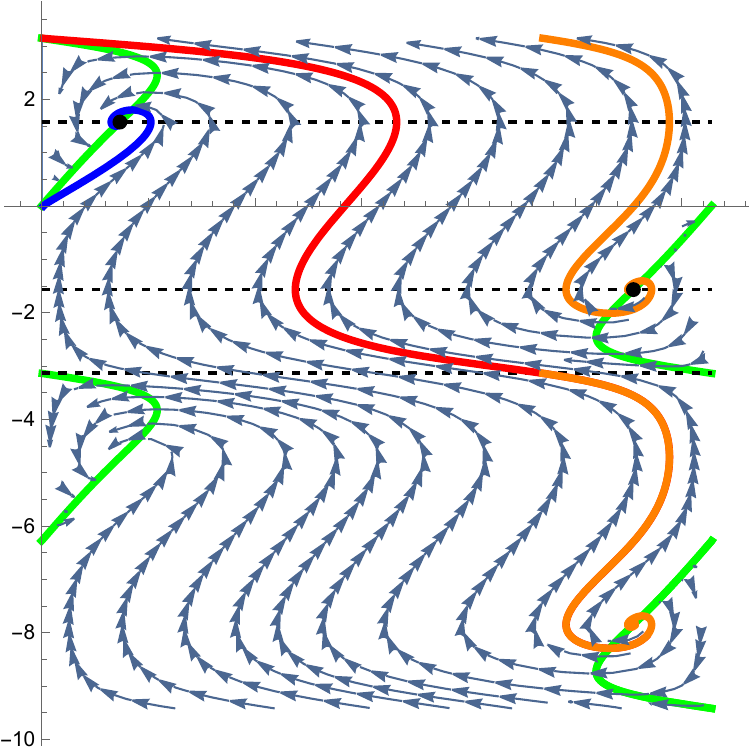}
\includegraphics[width=.4\textwidth]{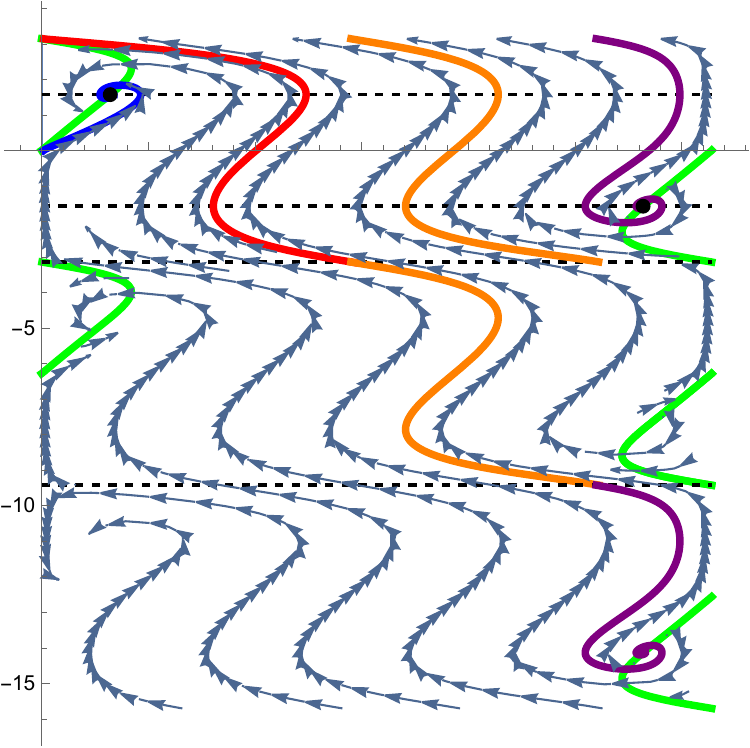}
\end{center}
\caption{The phase plane for $\lambda>1$ and the possible behaviors for the orbit $\gamma_-$. Left: $\lambda=1'3$, $\gamma_-$ reaches the line $\theta=-\pi$ and emerges from $e_1-(0,2\pi)$. Right, $\lambda=1'5$, $\gamma_-$ reaches the line $\theta=-3\pi$ and emerges from $e_1-(0,4\pi)$.}
\label{figcase>1}
\end{figure}

If we are in the first case, then $\gamma_-\to e_1$ as $s\to-\infty$. If we are in the second case, we consider now the orbit $\sigma_2$ passing through the point $(x_2,\pi)$ and again it either converges to $e_1$ or intersects the line $\theta=-\pi$. In any case, we glue again $\gamma_-$ and $\sigma_2-(0,4\pi)$ at $(x_2,-3\pi)$ to form a larger orbit, denoted again by $\gamma_-$. Note that this proceeding generates an strictly sequence $\{x_n\}$, consisting on the points of intersection of $\gamma_-$ with the lines $\theta=(1 -2n)\pi$. We claim that this process ends after a finite number of steps.

Arguing by contradiction, assume that $x_n\to x_\infty$ and if we name $\sigma_\infty$ the orbit passing through $(x_\infty,\pi)$, then $\sigma_\infty$ intersects the line $\theta=-\pi$, and in particular does not converge to $e_1$, and $\sigma_n\to\sigma_\infty$. Let assume that the endpoints of $\sigma_\infty$ are $(x_\infty^1,\pi)$ and $(x_\infty^2,0)$. Then, by Prop. \ref{propendpointsorbita} $x_\infty^1<x_\infty^2$. In particular, for $n$ big enough the endpoints of $\sigma_n$ also satisfy this property and hence $x_{n_0}>x_\infty$ for some $n_0$ big enough, a contradiction.

Consequently, $\gamma_-$ is glued with a final orbit $\sigma_{n_0},\ n_0\in\mathbb{N}$ and ends up converging to $e_1$; see Fig. \ref{figcase>1}, right, the orbit in purple. The corresponding surface $M_-$ has the topology of a disk, is properly immersed and the height function $z$ increases and decreases until its end converges to the CMC cylinder generated by $e_1$. Furthermore, the number of \emph{loops} that $M_-$ makes agrees with the number $n_0$ of orbits $\sigma_k$ needed for reaching the equilibrium $e_1$. See Fig. \ref{figperfiles}, left, the profile curve in red.

\textbf{Case $\lambda=1$}. Now, assume that $\lambda=1$. In this case, the curve $\Gamma$ also consists on two compact arcs, but this time $\Gamma(\pi)=\pi/2$ and similarly $\Gamma(-\pi)=\pi/2$. Hence, $\Gamma$ for $\theta\in(0,\pi)$ is a compact arc that joins the point $(0,0)$ and $(\pi/2,\pi)$, while in the region $\theta\in(-\pi,0)$ the curve $\Gamma$ is a compact arc that joins $(\pi,0)$ and $(\pi/2,-\pi)$. See Fig. \ref{figcase=1}, left.

\begin{figure}[h]
\begin{center}
\includegraphics[width=.4\textwidth]{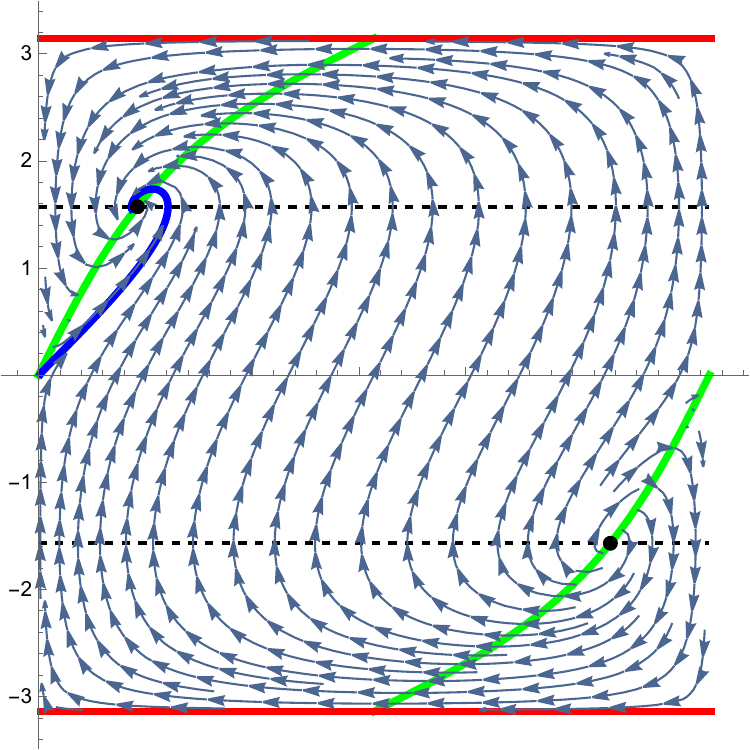}
\includegraphics[width=.4\textwidth]{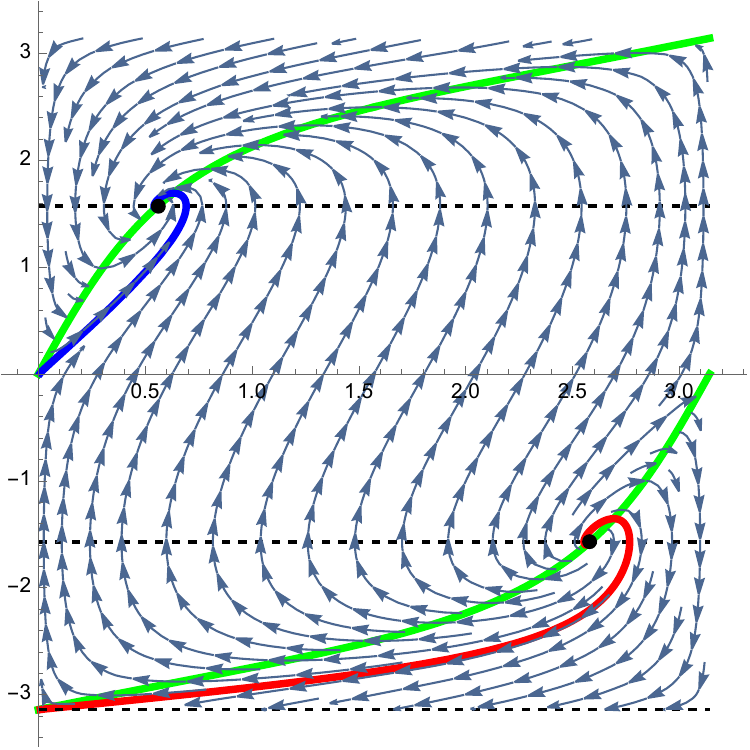}
\end{center}
\caption{Left: the phase plane for $\lambda=1$. Right: the phase plane for $\lambda<1$. There appear in both the curve $\Gamma$ in green, the orbit $\gamma_+$ in blue and $\gamma_-$ in red.}
\label{figcase=1}
\end{figure}

We begin with the $\lambda$-translator $M_-$, which is clear that agrees with a horizontal slice $\s^2\times\{z_0\},\ z_0\in\r$ with downwards orientation. Thus, $M_-$ determines the orbit $\gamma_-^1(s)=(s,\pi)$, which is the same as the orbit $\gamma_-^2(s)=(s,-\pi)$ by $2\pi$-periodicity. See Fig. \ref{figcase=1}, left, the orbits in red. In particular, by uniqueness of the Cauchy problem, the orbits $\gamma_-^1$ and $\gamma_-^2$ act as barriers and hence orbit in $\R$ must stay in the region $\theta\in(0,2\pi)$. Regarding the orbit $\gamma_+$, it has the point $(0,0)$ as endpoint and since it cannot intersect $\gamma_-^1$ neither spiral around $e_0$ converging to some closed limit cycle by Prop. \ref{propnonexistenceclosedorbits}, it must converge to $e_0$ as $s\to\infty$. See Fig. \ref{figcase=1}, left, the orbit in blue. The corresponding $\lambda$-translator $M_+$ has the same properties as the one for $\lambda>1$. In Fig. \ref{figperfiles} we see the profile curves for the orbit $\gamma_+$ in blue and for $\gamma_-$ in red.

\textbf{Case $\lambda<1$}. We finish with the case $\lambda<1$. Now we have to take into account that the function $\Gamma(\theta)$ changes of sign, precisely for $\theta_0=\arccos(-\lambda)$. For instance, we have two branches of the curve $\Gamma$ in the region $\theta\in(0,\pi)$; one for $\theta\in(0,\theta_0)$ and the other for $\theta\in(\theta_0,\pi)$, and the latter is given by the graph $x=\pi+\Gamma(\theta)$. The curve $\Gamma$ in the region $\theta\in(-\pi,0)$ is obtained similarly, or by the anti-symmetry properties of $\R$. See Fig. \ref{figcase=1}, right.

The properties of the orbit $\gamma_+$ having the point $(0,0)$, as well as the $\lambda$-translator $M_+$ are as in the previous cases, hence we skip the details. Now, let us focus on the orbit $\gamma_-$ corresponding to the $\lambda$-translator $M_-$. It has the point $(0,\pi)$ as endpoint, say at the instant $s=0$, but for $s<0$ the angle $\theta(s)$ increases. Geometrically, $M_-$ reaches the rotation axis with decreasing height function. Thus, we take advantage of the $2\pi$-periodicity of $\R$ and consider $\gamma_-$ as the orbit having $(0,-\pi)$ as endpoint at $s=0$. By monotonicity, it is clear that $\gamma_-$ cannot reach the antipodal fiber $\{(0,0,-1)\}\times\r$ and thus the only possibility for $\gamma_-$ is to converge to $e_1$ as $s\to-\infty$. Therefore, $M_-$ is a proper disk with strictly monotonous height function, hence embedded, and whose end converges to the CMC cylinder generated by $e_1$. See Fig. \ref{figperfiles}, right, the profile curve in red.

This concludes the classification of the rotational $\lambda$-translators intersecting the rotation axis.
\end{proof}


\subsection{Proof of Thm. \ref{t2}}

Now we classify the rotational $\lambda$-translators that do not intersect the rotation axis. This will conclude the classification of all rotational $\lambda$-translators in $\s^2\times\r$.

\begin{proof}

The proof follows almost immediately from the properties of the phase plane deduced in the proof of Thm. \ref{t1}, hence most of the details are skipped. We begin by the simpler case and treat the cases $\lambda<1$ and $\lambda=1$ at once. Thus, assume $\lambda\leq1$, fix $x_0\in(0,\pi)$ and let $\gamma$ the orbit passing through $(x_0,0)$ at $s=0$. By monotonicity and since in the case $\lambda=1$ the orbit $\gamma$ cannot intersect the orbit $\gamma_-(s)=(s,\pi)$, we deduce that $\gamma$ necessarily intersects the curve $\Gamma$ (this is clear if $\lambda<1$). Since $\gamma$ cannot converge to the axis $x=0$, neither stay at positive distance to $e_0$ we already know from previous discussions that $\gamma$ converges to $e_0$ as $s\to\infty$. Now, for $s<0$ the orbit $\gamma$ intersects the line $\theta=-\pi/2$. Furthermore, if $\lambda=1$ then $\gamma$ cannot intersect $\gamma_-(s)=(s,-\pi)$, while if $\lambda<1$ then $\gamma$ must intersect the curve $\Gamma$ and then cannot intersect $\gamma_-$. In any of both cases, $\gamma$ converges to $e_1$ as $s\to-\infty$. See Fig. \ref{figfasesnointerseca}, the orbit in red, center for $\lambda=1$ and right for $\lambda<1$. The corresponding $\lambda$-translator has the topology of an annulus, one end converges to the cylinder $e_0$ and the other converges to $e_1$. In between, the height function reaches a global minimum and a single loop occurs, hence the $\lambda$-translator is not embedded. See Fig. \ref{figperfilnointerseca}, center for $\lambda=1$ and right for $\lambda<1$.

\begin{figure}[h]
\hspace{-1cm}\begin{tikzpicture}[scale=1]
\node[anchor=south east,inner sep=0] at (-3,0){\includegraphics[width=.4\textwidth]{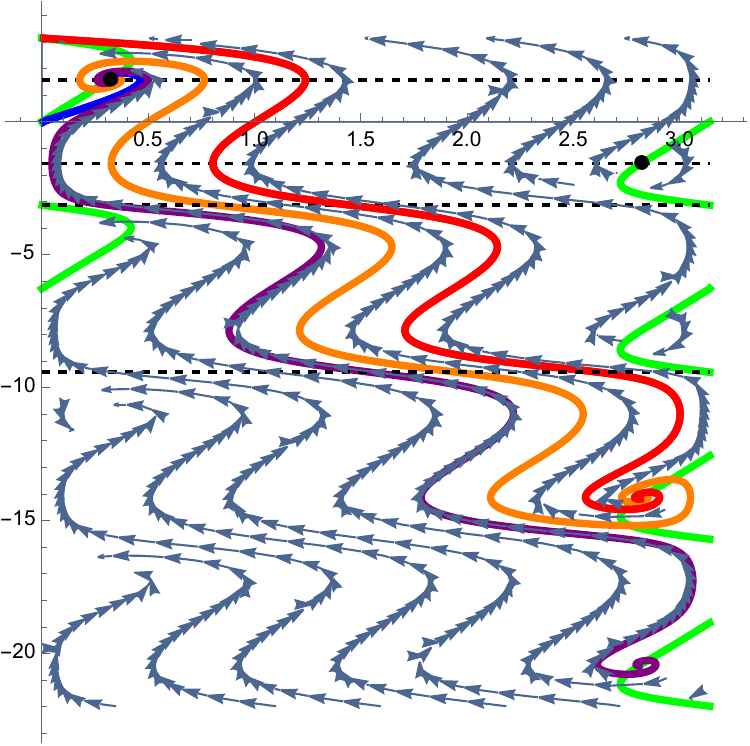}};
\node[anchor=south,inner sep=0] at (0,0){\includegraphics[width=.4\textwidth]{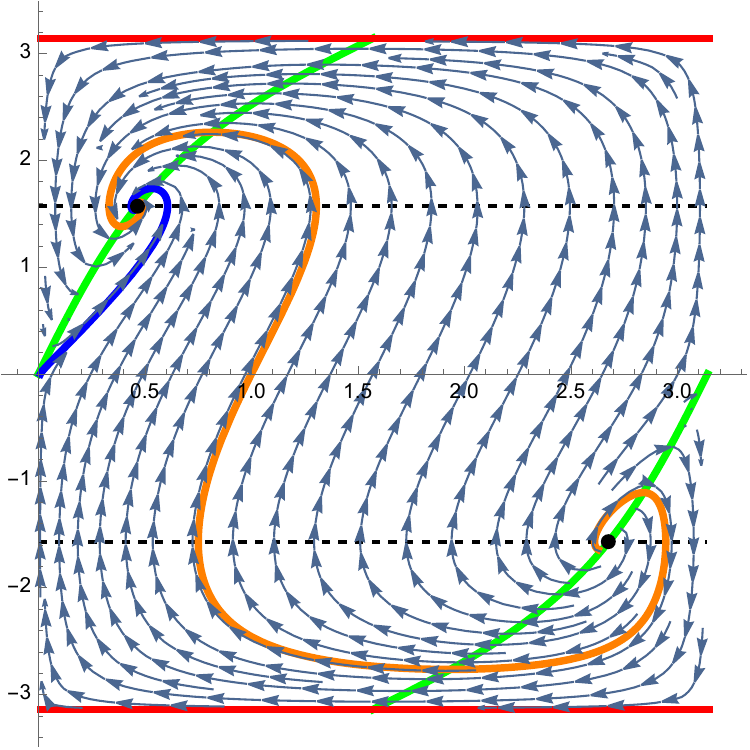}};
\node[anchor=south west,inner sep=0] at (3,0){\includegraphics[width=.4\textwidth]{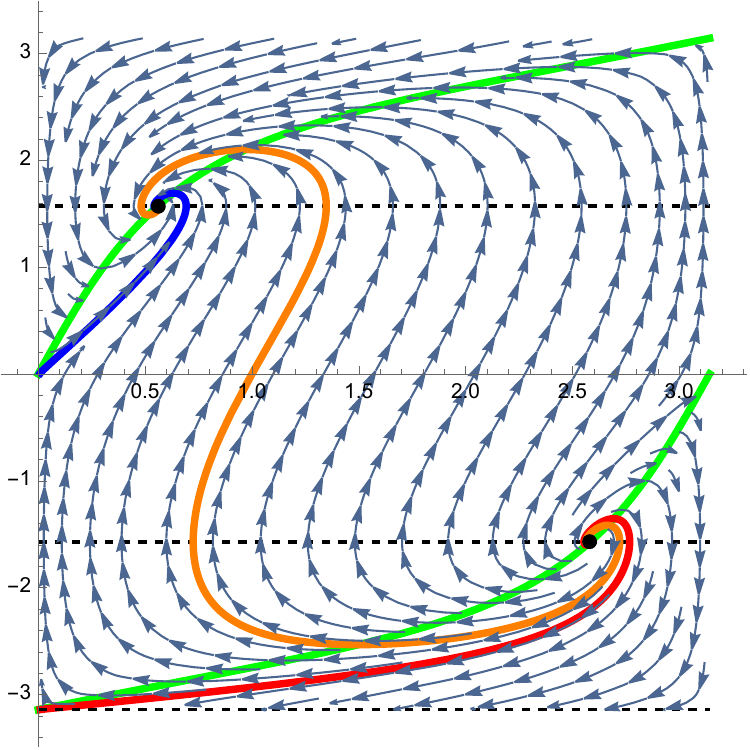}};
\end{tikzpicture}
\caption{The phase planes for $\lambda>1$ (left), $\lambda=1$ (center) and $\lambda<1$ (right), and an orbit in orange corresponding to a rotational $\lambda$-translator not-intersecting the rotation axis.}
\label{figfasesnointerseca}
\end{figure}

Finally, let $\lambda>1$ and fix $x_0\in(0,\pi)$. Denote by $(x_-,0),\ x_-\in(0,\pi)$, the point of intersection of $\gamma_-$ with $\theta=0$ and let $n_0\in\mathbb{N}$ the number of loops of $M_-$. In particular, $\gamma_-(s)\to e_1-(0,2n_0\pi)$. Assume $x_0<x_-$ and let $\gamma$ be the orbit passing through $(x_0,0)$ at $s=0$. For $s>0$ the orbits $\gamma_+$ and $\gamma_-$ act as barriers and $\gamma(s)\to e_0$ as $s\to\infty$. For $s<0$ and since $\gamma_-$ acts as a barrier again, either $\gamma(s)\to e_1-(0,2n_0\pi)$ as $s\to-\infty$, or there exists $n_1>n_0$ such that $\gamma(s)\to e_1-(0,2n_1\pi)$ as $s\to-\infty$. Since the number of loops of the rotational $\lambda$-translator generated by $\gamma$ is either $n_0$ or $n_1$, we conclude that in any case is greater than the number of loops of $M_-$. See \ref{figfasesnointerseca}, left, the orbit in orange. The corresponding profile curve self-intersects a finite number of times, has one end converging to the cylinder $e_0$ and the other converging to $e_1$. See Fig. \ref{figperfilnointerseca}, left.

\begin{figure}[H]
\begin{center}
\includegraphics[width=.8\textwidth]{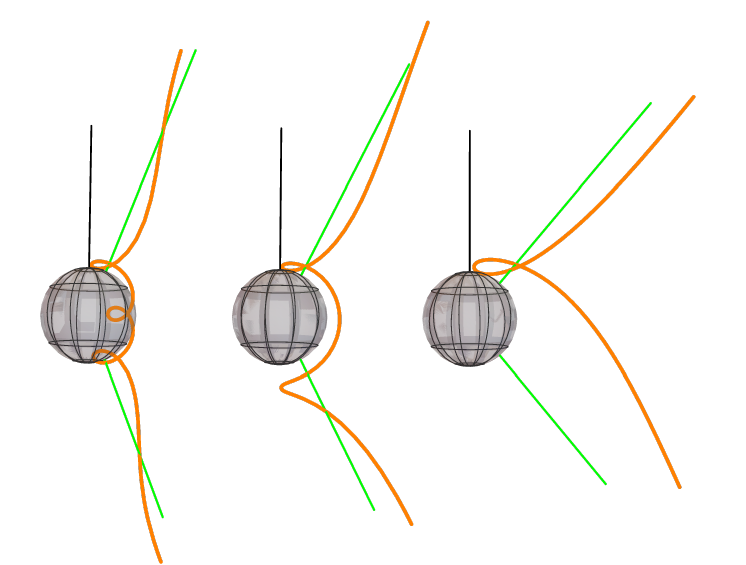}
\end{center}
\caption{The profile curve of a rotational $\lambda$-translator not-intersecting the rotation axis for $\lambda>1$ (left), $\lambda=1$ (center) and $\lambda<1$ (right).}
\label{figperfilnointerseca}
\end{figure}

\end{proof}

\subsection*{Ethics declarations}

Conflict of interest. The author have no conflict of interest to declare that are relevant to the content of this article. No data were used to support this study

\subsection*{Acknowledgment}

Partially supported by the grant PID2021-124157NB-I00, funded by MCIN/ AEI/10.13039/501100011033/ "ERDF A way of making Europe".


\noindent
Universidad Murcia. \\ 
\emph{E-mail address:} jabueno@um.es

\end{document}